\documentclass[12pt,intlimits]{amsart}

\usepackage{amssymb}

\usepackage{enumitem}

\makeatletter
\@namedef{subjclassname@2020}{%
  \textup{2020} Mathematics Subject Classification}
\makeatother

\allowdisplaybreaks

\usepackage[T1]{fontenc}

\newtheorem{theorem}{Theorem}[section]
\newtheorem{corollary}[theorem]{Corollary}
\newtheorem{lemma}[theorem]{Lemma}
\newtheorem{proposition}[theorem]{Proposition}
\newtheorem{problem}[theorem]{Problem}

\newtheorem{mainthm}[theorem]{Main Theorem}

\theoremstyle{definition}
\newtheorem{definition}[theorem]{Definition}
\newtheorem{remark}[theorem]{Remark}

\newtheorem*{xrem}{Remark}

\numberwithin{equation}{section}

\DeclareMathOperator{\len}{length}

\newcommand{\obj}[3]{\mathcal{F}^{#1}\mathbb{S}^{#2}\mathbf{G}_{#3}}

\newcommand{\R}{\mathbb R}
\newcommand{\N}{\mathbb N}

\newcommand{\subR}{{\mathbb R}}
\newcommand{\subRn}{{{\mathbb R}^n}}

\DeclareMathOperator{\supp}{supp}

\DeclareMathOperator*{\esssup}{ess\,sup}

\DeclareMathOperator{\osc}{\large{osc}}

\newcommand{\pp}{{p(\cdot)}}

\newcommand{\Lp}{L^{p(\cdot)}}

\newcommand{\Pp}{\mathcal P}

\newcommand{\D}{\mathcal{D}}

\newcommand{\F}{\mathcal{F}}

\newcommand{\B}{\mathcal{B}}

\DeclareMathAlphabet\EuRoman{U}{eur}{m}{n}
\SetMathAlphabet\EuRoman{bold}{U}{eur}{b}{n}
\renewcommand{\mathsf}{\EuRoman}

\def\Xint#1{\mathchoice 
	{\XXint\displaystyle\textstyle{#1}}%
	{\XXint\textstyle\scriptstyle{#1}}%
	{\XXint\scriptstyle\scriptscriptstyle{#1}}%
	{\XXint\scriptscriptstyle\scriptscriptstyle{#1}}%
	\!\int} 
\def\XXint#1#2#3{{\setbox0=\hbox{$#1{#2#3}{\int}$} 
		\vcenter{\hbox{$#2#3$}}\kern-.5\wd0}}

\def\avgint{\Xint-}

\usepackage{commath}
\usepackage{microtype}
\usepackage{mathtools}

\begin{document}


\baselineskip=17pt


\title[Sobolev meets Riesz]{Sobolev meets Riesz: a characterization of weighted Sobolev spaces via weighted Riesz bounded variation spaces}

\author[D. Cruz-Uribe, OFS]{David Cruz-Uribe, OFS}
\address[D. Cruz-Uribe,OFS]{Departament of Mathematics\\ University of Alabama \\ Tuscaloosa, USA}
\email{dcruzuribe@ua.edu}

\author[O. Guzm\'an]{Oscar Guzm\'an}
\address[O. Guzm\'an]{Departament of Sciences and Humanities\\ University of America \\ Bogotá, Colombia and Department of Mathematics \\ United Arab Emirates University \\ Al Ain, United Arab Emirates}
\email{oscar.guzman@profesores.uamerica.edu.co}

\author[H. Rafeiro]{Humberto Rafeiro}
\address[H. Rafeiro]{Department of Mathematics \\ United Arab Emirates University \\ Al Ain, United Arab Emirates}
\email{rafeiro@uaeu.ac.ae}

\date{}

\begin{abstract}
  We introduce weighted Riesz bounded variation spaces
  defined on an open subset of the $n$-dimensional Euclidean space 
  and use them to characterize weighted Sobolev spaces when the weight
  belongs to the Muckenhoupt class.  As an application, using Rubio de
  Francia's extrapolation theory, a
  similar characterization of the variable exponent Sobolev spaces via
  variable exponent Riesz bounded variation spaces is obtained.
\end{abstract}


\maketitle

	\section{Introduction}
	\label{sec:intro}
	In the foundational paper \cite{Riesz1910}, F. Riesz proved
        (in modern terminology) that 
        an absolutely continuous
        function $f: I \to \subR$ 
        belongs 
	to the Sobolev space $W^{1,p}(I)$, $1 < p < \infty$ and $I \subset \subR$ an interval, if and only if
	\begin{equation}\label{eq:RieszVardefinition}
		\sup\sum_{j}\frac{|f(x_{j})-f(x_{j-1})|^{p}}{|x_{j}-x_{j-1}|^{p-1}}<\infty, 
	\end{equation}
	where the supremun is taken over all finite
        partitions 
	of $I$. The quantity in \eqref{eq:RieszVardefinition} is
        called the \emph{Riesz $p$-variation of $f$ on $I$}. We
        refer the interested readers to \cite{Appell-Merentes} for a
        comprehensive survey of the classical theory of bounded variation spaces.
	There has been progress  on
        generalizations of  Riesz bounded variation spaces, 
        for example, 
     weighted Riesz spaces~\cite{MR4154132}, and 
 variable exponent Riesz spaces \cite{castillo2016variable,Castillo2019,
   Kakochashvili_OnRiesz_2016}, to name just two.

 More recently, work has been done on extending the notion of Riesz bounded
 variation spaces to the case of  functions defined
 on general domains in $\R^n$, $n>1$, see, for instance, the
 works of Angeloni~\cite{Angeloni2017}, Barza and
 Lind~\cite{Barza-Lind2015}, 
 Bojarski~\cite{bojarski2011remarks}, and Mal\'y~\cite{Maly1999}.  We
 are particularly interested in \cite{Barza-Lind2015} since the space  
 $RBV^p(\Omega)$, $\Omega \subset \mathbb R^n$, is introduced and is defined as the set of all functions $f$ such that
 \[ V_p(f; \Omega) = \sup \bigg[ \sum_{B_k\in \D} \bigg(
   \frac{\osc_{B_k}(f)}{r_k}\bigg)^p |B_k| \bigg]^{1/p} < \infty, \]
 where the supremum is taken over all countable
              collections $\D=\{B_k\}_{k=1}^\infty$ of disjoint balls of radius
              $r_{k}$ contained in $\Omega$.  It is shown that if
              $p>n$, then $f\in
              RBV^p(\Omega)$ if and only if  $f \in W^{1,p}(\Omega)$. 

              The goal of this paper is to extend this result and
              prove a version of the Riesz theorem characterizing
              weighted Sobolev spaces in $\R^n$,
              $W^{1,p}\left(\Omega, w\right)$, for weights $w$ in the
              Muckenhoupt class $A_p$, using an appropriate weighted version of $RBV^p$.  To state our main results, we give some
              essential definitions; we defer technical details and
              some additional definitions until the next section.
              Hereafter, $\Omega \subset \R^n$ will be an open set.
              Given a weight $w\in A_\infty$, 
              let
              \begin{equation}\label{eq:rw}
                r_w := \inf\{q>1 : w\in A_q\}.
\end{equation}
                For
              $1\leq p<\infty$, let $W^{1,p}\left(\Omega, w\right)$
              denote the collection of functions
              $f\in L^{p}(\Omega,w) $ whose weak derivatives $D_{j}f$
              belong to $L^{p}\left(\Omega,w\right)$,
              $1\leqslant p <\infty$.  The space $W^{1,p}(\Omega,w)$
              is the collection of functions $f$ with weak derivatives
              endowed with the norm
	\begin{equation}\label{def:Sobolev}
          \|f\|_{W^{1,p}(\Omega,w)}
          =  \left(\int_{\Omega}| f(x)|^{p}w(x)\dif{x}\right)^{1/p}
          +\left(\int_{\Omega}|\nabla f(x)|^{p}w(x)\dif{x}\right)^{1/p}.
              \end{equation}
Weighted Sobolev spaces were introduced in~\cite{MR643158} and have
applications to the study of degenerate PDEs.  
              
              Given a measurable function $f : \Omega \rightarrow\R$ and $1\leq
              p<\infty$, we say that
              $f$ is of weighted Riesz bounded $p$-variation on
              $\Omega$, denoted by $f\in RBV^p(\Omega,w)$, if
              \[ 	V_{p}\left(f;\Omega, w\right)
                \coloneqq
                \sup \bigg[\sum_{B_k \in
                    \D}\left(\frac{\osc_{B_k}(f)}{r_{k}}\right)^{p}w(B_k)\bigg]^{1/p}
                <\infty. \]
             
Our first main result, whose proof we postpone to Section \ref{sec:proof main theorem}, is the following. 
	\begin{theorem}\label{theo:RBVpq-W1p}
          Given $\Omega\subseteq \subRn$ an open set, assume that
          $p>n$. Let $w\in A_p$ be such that $p>n r_w$, with $r_w$ defined in \eqref{eq:rw}. Then
          $f\in W^{1,p}\left(\Omega,w\right)$ if and only if the function $f$ is
          continuous (perhaps after being redefined on a null set) 
          and $f\in RBV^{p}\left(\Omega,w\right)\cap L^p(\Omega,w)$. Furthermore,
		\begin{equation}\label{eq:Main_Inequality}
                  \|\nabla f\|_{L^{p}\left(\Omega,w\right)}
                  \lesssim  V_{p}\left(f;\Omega, w\right)
                  \lesssim \|\nabla f\|_{L^{p}\left(\Omega,w\right)},
		\end{equation}
		where the implicit constants depend on $p, n$, and $[w]_{A_{p}}$. 	
              \end{theorem}
              
	\begin{remark}
          The unweighted version of Theorem~\ref{theo:RBVpq-W1p}
          (i.e., when $w=1$ so $r_w=1$) 
          was
          proved by Barza and Lind in
          \cite{Barza-Lind2015}.   We note in passing an omission in
          the statement of their main result:  the
          hypothesis  that $f\in L^p(\Omega)$ is missing.  It is straightforward
          to give a counter-example without this assumption.  Let
          $\Omega=\R^n$ and let $f\equiv 1$.  Then $f$ is
          continuous and $f \in RBV^p(\R^n)$ for any $p>n$, but
          $f \not\in W^{1,p}(\R^n)$. 
	\end{remark}

        As a corollary to the proof of Theorem~\ref{theo:RBVpq-W1p}, we
        have that the left-hand side  inequality
        in~\eqref{eq:Main_Inequality} is always true.

        \begin{corollary} \label{cor:RBV-embed}
          Let $1\leq p<\infty$ and $w$ be a weight.  If $f\in
          RBV^p(\Omega,w)$, then $\|\nabla f\|_{L^p(\Omega,w)} \lesssim
          V_p(f; \Omega, w)$.
        \end{corollary}
        
        \begin{remark}
          Although we are primarily concerned with spaces defined with
              $A_{p}$ weights, Corollary~\ref{cor:RBV-embed} and
              Corollary~\ref{cor:doubling} below show that
              some of our results are true with
              weaker assumptions on the weights.  Many classical
              results for Sobolev spaces can be extended to the
              weighted case by using a wider class of weights: e.g.,  see
              \cite[Chapter 1]{Tero} for a discussion of Sobolev
              spaces defined using the so-called $p$-admissible weights.  It is
              an interesting question to determine the weakest
              hypotheses on the weights to define a \emph{rich} theory of weighted
              Riesz
              $p$-variation. 
            \end{remark}

            We can also give a weak-type estimate for the local
            Lipschitz constant of a function.  Recall that a function
            $f$ is said to be \textit{Lipschitz continuous} at $x$ if
	\[
	L_{f}(x)=\varlimsup_{y\rightarrow x}\frac{|f(x)-f(y)|}{|x-y|}<\infty.
	\]
        
	\begin{theorem}\label{prop:Main-Proposition}
		Assume that $w\in A_{p}$, $1\leq p<\infty$, and $f\in RBV^p(\R^n,w)$. Then for all
                $t>0$, 
		\[
                  w\left(\{x\in \subRn :  L_{f}(x)>t\}\right)
                  \lesssim \left(\frac{V_{p}(f; \R^n,w)}{t}\right)^{p}.
		\]
                 The implicit constants depend on $p, n$, and $[w]_{A_{p}}$. 	
              \end{theorem}

As a corollary to the proof of Theorem~\ref{prop:Main-Proposition}, we
have that the $A_p$ condition can be significantly weakened.  Given any ball $B=B(x,r)$, let
$2B:=B(x,2r)$.   We say that a weight is \emph{doubling} if $w(2B) \leq Cw(B)$ with a constant $C$ independent of $B$.

 \begin{corollary} \label{cor:doubling} 
Theorem~\ref{prop:Main-Proposition}  remains true   
when $w$ is only assumed to be doubling.
\end{corollary}

A straightforward consequence of
              Theorems~\ref{theo:RBVpq-W1p}
              and~\ref{prop:Main-Proposition} is the following result.

\begin{corollary}\label{cor:Almost-Differentiability-W1p}
  Let $w\in A_p$ and  $1\leq p<\infty$. If $f\in RBV^p(\R^n,w)$, then
  $f$ is differentiable almost everywhere.  Moreover, if  $p>nr_w$ and $f\in W^{1,p}(\subRn, w)$,
then $f$ is differentiable almost everywhere.
\end{corollary}

\begin{remark}
  In the unweighted case, the differentiability of functions in
  $W^{1,p}(\R^n)$, $p>n$, is due to Calder\'on~\cite{MR0045200} (see
  also~\cite[Theorem~6.17]{Heinonen2001}).   The unweighted version of
  Corollary~\ref{cor:Almost-Differentiability-W1p} was proved by Barza
  and Lind~\cite{Barza-Lind2015}; for the proof of the
  differentiability of functions in $RBV^p(\R^n)$ they assume that $p>n$, 
  but this is not needed in the proof.
\end{remark}

When $p=n$, Mal\'y~\cite{Maly1999} defined a space of functions of
bounded $n$-variation that is equivalent to $RBV^n(\Omega)$ and proved
that it is continuously embedded in $W^{1,n}(\Omega)$.  Furthermore, he
showed that the elements of $RBV^n(\Omega)$ are differentiable almost
everywhere.  By modifying the proofs of Theorem~\ref{theo:RBVpq-W1p}
and~\ref{prop:Main-Proposition} we can further extend this result to the
weighted case.

 \begin{theorem}\label{thm:embedding-n-variation-sobolev}
   Let $w \in A_n$ and 
   $f\in RBV^{n}(\Omega,w)\cap L^n(\Omega, w)$. Then
   $f\in W^{1,n}(\Omega,w)$ with $\|\nabla f\|_{L^p(\Omega,w)}\lesssim
   V_n(f; \Omega, w)$.  Moreover, $f $
   is differentiable almost everywhere.
 \end{theorem}

 \medskip
 
        As an application of Theorem~\ref{theo:RBVpq-W1p} and the
        Rubio de Francia extrapolation theory, we extend this result
        to the variable exponent Sobolev spaces.
        Define the variable Lebesgue space $\Lp(\Omega)$ to be all
        functions $f$ such that
        \begin{equation} \label{eqn:var-norm}
          \|f\|_{L^{p(\cdot)}(\Omega)}
          =\inf  \Big\{\lambda>0 :
          \rho_{p(\cdot),\Omega}\left( f/\lambda \right)\leqslant 1 \Big\},
              \end{equation}
	where 
        \begin{equation} \label{eqn:var-modular}
	\rho_{p(\cdot),\Omega}(f):=\int_{\Omega}|f(x)|^{p(x)}\dif{x}.
      \end{equation}
      The variable exponent Sobolev space $W^{1,\pp}(\Omega)$ consists of all functions $f\in
\Lp(\Omega)$ that are weakly differentiable and $\nabla f \in
\Lp(\Omega)$.
        We say $\pp \in LH(\Omega)$, $\pp : \Omega \rightarrow  [1,\infty)$,  if it is log-H\"older continuous
        locally and at infinity, i.e. it satisfies  \eqref{eq:local_log_Holder} and
        \eqref{eq:infinity_log_Holder}, respectively.
        
The space of functions of bounded $\pp$-variation,
$RBV^\pp(\Omega)$, is intuitively defined as the space
$RBV^p(\Omega)$, but with the constant exponent $p$ replaced by a
variable exponent $\pp$.  We defer the precise definition to
Section~\ref{sec:Variable-expoent-section}.  

\begin{theorem}\label{theo:Riesz_Variable-Exponent}
  Given an open set $\Omega \subset \R^n$, let $p(\cdot)\in
  LH(\Omega)$, $n<p_-\leqslant p_+<\infty$. Then 
  $f\in W^{1,p(\cdot)}(\Omega)$ if and only if $f$ is  a continuous function (possibly after being
                redefined on a null set)  and
                $f \in RBV^{p(\cdot)}(\Omega)\cap \Lp(\Omega)$.
                Furthermore,
		\begin{equation}\label{eq: var-exponent_Riesz}
			\|f\|_{RBV^{p(\cdot)}(\Omega)} \approx \|\nabla f\|_{p(\cdot),\Omega}.
                      \end{equation}
	\end{theorem}

              \medskip

	The remainder of this paper is organized as follows.  In
        Section \ref{Sec:Pre} we give some definitions and auxiliary
        results regarding $A_{p}$ weights, define and give some basic
        properties of variable Lebesgue spaces, and define the
        weighted Riesz bounded variation spaces
        $RBV^{p}(\Omega,w)$. Some embedding results in this new scale
        of functions are established as well. In
        Section~\ref{sec:proof main theorem} we prove
        Theorems~\ref{theo:RBVpq-W1p},~\ref{prop:Main-Proposition},
        and~\ref{thm:embedding-n-variation-sobolev}. Finally, in
        Section~\ref{sec:Variable-expoent-section} we define the
        variable exponent Riesz bounded variation space
        $RBV^{p(\cdot)}(\Omega)$ and prove
        Theorem~\ref{theo:Riesz_Variable-Exponent}.

	\section{Preliminaries}\label{Sec:Pre}

        Throughout this paper $C,c$ will denote constants, depending
        on the underlying parameters, that may
        change their values even from line to line. Given $A,B>0$, we
        write $A\lesssim B$ if there exists $C>0$ such that
        $A\leqslant CB$. Additionally, if $A\lesssim B$ and
        $B\lesssim A$ simultaneously, we write $A\approx B$.  In the
        sequel we regard $\Omega$ as an open subset of $\subRn$. For
        an open ball with center $x$ and radius $r>0$ we write
        $B(x,r)$. Throughout this paper all the cubes $Q\subset \R^n$
        will have their sides parallel to the coordinate axis.  We
        denote the integral average of $f$ over the measurable subset
        $E$ by
	\[ \langle f \rangle_E := \frac{1}{|E|}\int_E f(x) \dif  x = \avgint_E f(x)\dif x ,\,\, |E|>0,\]
	where $|E|$ denotes the Lebesgue measure of the set $E$.

	\subsection{Muckenhoupt $A_p$ weights}
        By a weight we mean a non-negative, locally integrable
        function.  The measure of $E$ induced by the weight $w$ is  $w(E):=\int_{E}w(x)\dif{x}$.
	
A weight $w$ belongs to the Muckenhoupt $A_p$ class, $1<p<\infty$,
denoted $w\in A_p$,  if 
	\[[w]_{A_{p}}=\sup_Q \avgint_Q w\dif x \bigg(\avgint_Q w^{1/(1-p)}\dif x\bigg)^{p-1} <
	\infty, \]
	where the supremun is taken over all cubes (or balls) in
        $\subRn$. When $p=1$, we say that $w\in A_1$ if
        \[ [w]_{A_1} = \sup_Q \esssup_{x\in Q} w(x)^{-1}\avgint_Q
          w\dif x < \infty.  \]
The class $A_{\infty}$ is defined by
	\[
	A_\infty= \bigcup_{p\geqslant 1}A_p.
      \]
      For more information on $A_p$ weights, and for proofs of the
      following results which we will need below, 
      see~\cite{JavierDuBook,grafakosmodern}.
      It follows from H\"older's inequality and the definition that if
      $1\leqslant p_1<p_2<\infty$, then $ A_{p_1} \subset A_{p_2}$.
      Moreover, the $A_p$ classes are left-open:  if $w\in A_p$ for
      some $p>1$, then there exists $\varepsilon>0$ such that $w\in
      A_{p-\varepsilon}$  (see~\cite[Corollary~7.6]{JavierDuBook}).
      Consequently, recalling that we defined
      \[ r_w = \inf\{ q>1 : w \in A_q \}, \]
      if $w\in A_p$, then $r_w<p$, and for every $q>r_w$, $w\in A_q$.

	\begin{lemma}\label{lemma: GeneralDoublingCon}
		Assume that  $w\in A_p$, $1\leqslant p < \infty$. Then 
		for every measurable subset $E$ of $Q$, 
		\begin{equation*}
			\left(\frac{|E|}{|Q|}\right)^{p}\leqslant [w]_{A_p}\frac{w(E)}{w(Q)}.	   
		\end{equation*}       
	\end{lemma}  
	
	We say that $w\in RH_{s}$, for some $s>1$, if 
	\begin{equation}\label{eq:RHIdef}
		[w]_{RH_{s}} = \sup_Q \frac{ \langle w ^s \rangle_Q ^{1/s}  }{ \langle w  \rangle_Q} < \infty, 
	\end{equation}
	where the supremun is taken over all cubes  in $\subRn$.
        
	\begin{lemma}\label{lemma:Ainfty_RHS}
		Assume that $w\in A_{\infty}$. Then there exists $s>1$, depending on $[w]_{A_p}$, such that $w \in RH_s$.
	\end{lemma}
	

	Finally, for $1\leq p<\infty$, $W^{1,p}\left(\Omega,w\right)$ is a Banach
      space (see \cite[Proposition 2.1.2]{Turesson}). Moreover, when
      $w\in A_p$, $W^{1,p}\left(\Omega,w\right)$ is the
      completion of
      $C^{\infty}(\Omega)\cap W^{1,p}\left(\Omega,w\right)$ with
      respect to the norm \eqref{def:Sobolev} (see ~\cite[Theorem
      2.5]{Tero}).

      \subsection{Variable exponent spaces}
            In this section, we recall the  definition of variable Lebesgue and Sobolev spaces.
      For further information on these spaces, 
      see~\cite{cruz-fiorenza-book,diening-harjulehto-hasto-ruzicka2010}.
      Let $\Pp(\Omega)$ denote the collection of all measurable exponent
      functions  $\pp: \Omega \to [1,\infty)$.  
      Denote the essential supremum and infimum of $\pp$ on a set
      $E\subseteq \subRn$ by $p_{+}(E)$ and $p_{-}(E)$,
      respectively. For the sake of simplicity, we write $p_{-}$ and
      $p_{+}$ when $E=\subRn$.  Hereafter, we will generally assume that
      $p_+<\infty$.   Recall the definition of the
      norm~\eqref{eqn:var-norm} and modular~\eqref{eqn:var-modular}
      associated to the exponent $\pp\in \Pp(\Omega)$.    The variable Lebesgue space
      $L^{p(\cdot)}(\Omega)$ is the collection of all measurable
      functions such that $\|f\|_{\Lp(\Omega)}<\infty$.  When there is
      no confusion about $\Omega$, we will sometimes write $\|f\|_\pp$
      for the norm. 

       Given $\pp\in \mathcal{P}(\Omega)$, we say that $\pp$ satisfies the
       local log-H\"older  continuity condition, and denote this by
       $\pp \in LH_{0}(\Omega)$  if there exists a constant $C_0$ such that
	\begin{equation}\label{eq:local_log_Holder}
		|p(x)-p(y)|\leqslant \frac{C_0}{-\log(|x-y|)},\quad x,y \in \Omega,\quad |x-y|<1/2,
	\end{equation}
	and we say that $\pp$ is  log-H\"older continuous at infinity,
        denoted by $\pp\in LH_{\infty}(\Omega)$ if there exist $p_{\infty}$ and $C_{\infty}>0$ such that
	\begin{equation}\label{eq:infinity_log_Holder}
		|p(x)-p_{\infty}|\leqslant \frac{C_{\infty}}{\log(e+|x|)},\quad x\in\Omega.
              \end{equation}
              When $\pp$ satisfies \eqref{eq:local_log_Holder} and
              \eqref{eq:infinity_log_Holder} we say that $\pp$ is a
              log-H\"older continuous function and denote this by
              $\pp \in LH(\Omega)$. It is well known that the $LH$
              condition is a natural assumption in the variable
              exponent setting. For instance, $LH$ regularity is
              sufficient for the maximal operator to be bounded on
              $L^{p(\cdot)}(\subRn)$, see,  for example,~
              \cite{cruz-fiorenza-book,diening-harjulehto-hasto-ruzicka2010}.

        Given a set $E$,  $0<|E|<\infty$, the harmonic mean of $\pp$ on $E$ is given by 
	\begin{equation*}
		\frac{1}{p_{E}}= \frac{1}{|E|}\int_E \frac{\dif
                  x}{p(x)}. 
	\end{equation*}

        Given a function $f\in
        L^{p(\cdot)}(\Omega)$, it belongs to the Sobolev space
        $W^{1,p(\cdot)}(\Omega)$ if its weak derivatives $D_{j}f$,
        $1\leqslant j\leqslant n$, exist and belong to
        $L^{p(\cdot)}(\Omega)$.  $W^{1,p(\cdot)}(\Omega)$  is a Banach space when
        equipped with the norm
        $\|f\|_{W^{1,p(\cdot)}(\Omega)}=\|f\|_{L^{p(\cdot)}(\Omega)}+\|\nabla
        f\|_{L^{p(\cdot)}(\Omega)}$ (see~\cite[Theorem~6.6]{cruz-fiorenza-book}).
        Though $C^{\infty}(\Omega)\cap
        W^{1,p(\cdot)}(\Omega)$ need not be dense in
        $W^{1,p(\cdot)}(\Omega)$ for general functions $\pp\in
        \mathcal{P}(\Omega)$, it is  when $\pp\in LH(\Omega)$, cf.~\cite[Theorem~6.14]{cruz-fiorenza-book}).

	\subsection{Weighted Riesz bounded variation spaces}
        For  $f: \Omega \to \subR$,  its oscillation on a set  $E\subset \Omega$ is defined as
	\[
	\osc_{E}(f):= \sup_{x,y\in E}|f(x)-f(y)|.    
	\]
	Given a weight $w$ and $1\leq p<\infty$,  define the weighted
        Riesz $p$-variation of $f$ on the set $\Omega$ by
	\begin{equation}\label{def:Vpqspaces}
          V_{p}\left(f;\Omega,w\right)
          \coloneqq \sup \left[\sum_{B_k \in \D}
            \left(\frac{\osc_{B_k}(f)}{r_{k}}\right)^{p}w(B_k)\right]^{1/p},
	\end{equation}
	where the supremum is taken over all the countable collections
        $\D=\{B_k\}_{k=1}^\infty$ of disjoint balls of radius $r_{k}$ contained in
        $\Omega$. If $V_{p}\left(f;\Omega,w\right)<\infty$, we say
        that $f\in RBV^{p}\left(\Omega,w\right)$. 

\begin{remark} When $w=1$, 
        since
        $|B|\approx r^{n}$,  we recover the definition of Riesz variation given in
        ~\cite{Barza-Lind2015},  up to a dimensional constant.
      \end{remark}

       For
      functions $f$ and $g$, and $\alpha \in \R$, by Minkowski's inequality for sequence spaces, we have 
      \[ V_{p}(f+g; \Omega,w) \leq  V_{p}(f; \Omega,w)
        +  V_{p}(g; \Omega,w),
\]
        and
  \[      V_{p}(\alpha f; \Omega,w)
        = |\alpha| V_{p}(f; \Omega,w); \]
      hence         $V_{p}(f;\Omega,w)$ is a seminorm.  It is not a norm, since
        $V_{p}(f;\Omega,w)=0$ for every constant function
        $f$.

			\begin{lemma}\label{prop:Embedding-RBVp}
				Given $w$ a weight, assume that $\Omega$ is a bounded set. If $1<p_{1}<p_2<\infty$, then
				\[
				RBV^{p_2}\left(\Omega,w\right) \hookrightarrow RBV^{p_1}\left(\Omega,w\right).
				\]
                              \end{lemma}
                              
			\begin{proof}
                          Fix $f \in RBV^{p_2}\left(\Omega,w\right)$
                          and let $\D=\{B_k\}_{k=1}^\infty$ be a
                          disjont collection of countable balls
                          contained in $\Omega$.  Fix $s>1$ such
                          that $1/p_2+1/s=1/p_1$. Then, by H\"older's
                          inequality,
				\begin{align*}
					\sum_{B_k \in \D
                                  }\left(\frac{\osc_{B_k}(f)}{r_k}\right)^{p_1}\!\!\!w(B_k)
                                  &\leqslant \left(\sum_{B_k \in \D
                                    }{\left(\frac{\osc_{B_k}(f)}{r_k}\right)^{p_2}\!\!\!w(B_k)}\right)^{\frac{p_1}{p_2}}
\!                                    \left(\sum_{B_k \in \D }\!\!w(B_{k})\right)^{\frac{p_1}{s}}\\
					& \leqslant
                                   V_{p_2}\left(f;\Omega,w\right)^{p_{1}}w(\Omega)^{\frac{p_1}{s}}\\
                                  & <\infty.
				\end{align*}
			If we take the supremum over all such
                        collections $\D$, we get the desired embedding.
			\end{proof}

		
		\section{Proof of
                  Theorems~\ref{theo:RBVpq-W1p},~\ref{prop:Main-Proposition},
                and~\ref{thm:embedding-n-variation-sobolev}}\label{sec:proof main theorem} 
		
		\begin{proof}[Proof of Theorem \ref{theo:RBVpq-W1p}]
                  Fix $w\in A_p$, $p>nr_w$.  We first prove that if
                  $f \in W^{1,p}(\Omega,w)$, then
                  $f\in RBV^p(\Omega,w)$.  To do so, we prove
                  a weighted version of the well-known \emph{Morrey
                    inequality} (see \cite[Theorem
                  9.12]{BrezisBook}).  Suppose that
                  $f\in C^{\infty}(\Omega)\cap W^{1,p}(\Omega,w)$. 
                  Fix $x_0\in \Omega$ and $R>0$ such that the open ball
                  $B(x_0,2R)$ is contained in
                  $\Omega$.   Fix $y,z\in  B(x_0,R)$  and let  $B_r= B(x,r)
                  \subset B(x_0,2R)$ be any ball containing
                  $y$ and $z$.    Then
			\begin{equation*}
				f(z)-f(y)=\int_{0}^{1}\nabla f(tz-(1-t)y) \cdot (z-y)\dif t.
			\end{equation*}
			Observe that $B(tz+(1-t)y,tr)\subset B(x,r)$
                        and $|z_j-y_j|\leqslant r$.  Since $p>nr_w$,
                        we can fix $\delta>0$ such that
                        $n<1+\delta< \frac{p}{r_w}$; therefore, if
                        $q=\frac{p}{1+\delta}$, then $nr_w<nq < p$.
                        Hence, by Fubini's theorem, a change of
                        variables, and Hölder's inequality applied
                        twice, we have 
			\begin{align*}
                          &|f(z)-\langle f\rangle_{B_{r}}|
                          \leqslant \avgint_{B_{r}}\int_{0}^{1}\sum_{j=1}^{n}\abs{D_{j} f(tz-(1-t)y)}|z_{j}-y_j|\dif t\dif y\\
				&\lesssim
				\sum_{j=1}^{n}\frac{1}{r^{n-1}}\int_{0}^{1}
				\left(\int_{B_r}|D_jf(y)|^{1+\delta}\dif y\right)^{1/(1+\delta)}|B(tx+(1-t)y,tr)|^{\delta/(1+\delta)} \frac{\dif t}{t^{n}}\\
				&\lesssim
				\sum_{j=1}^{n}r^{1-\frac{n}{1+\delta}}\int_{0}^{1}\left(\int_{B_r}\abs{D_{j}f(y)}^{1+\delta}\dif y\right)^{1/(1+\delta)} \frac{\dif t}{t^{n/(1+\delta)}}\\
				&\lesssim r^{1-\frac{nq}{p}}\sum_{j=1}^{n}\left(\int_{B_r}|D_{j}f(y)|^{p}w(y)\dif y\right)^{1/p}\left(\int_{B_r}w^{1/(1-q)}(y)\dif y\right)^{(q-1)/p}.
			\end{align*}
                        Note that in the final estimate we use that
                        $n<1+\delta$ to evaluate the integral.  
			Hence, if we  set $\sigma := w^{1/(1-q)}$, the preceding estimate for $ z,y \in B_r $ leads to
			\begin{multline*}
 |f(z)-f(y)| \leqslant |f(z)-f_{B_r}|+|f(y)-f_{B_r}| \\
 \lesssim  r^{1-n\frac{q}{p}}\sigma(B_r)^{(q-1)/p}\|\nabla
 f\|_{L^{p}(B_r,w)}
 \leq  r^{1-n\frac{q}{p}}\sigma(B_{2R})^{(q-1)/p}\|\nabla f\|_{L^{p}(B_R,w)}.
\end{multline*}
Moreover, since $w\in A_q$, then $\sigma(B_R)<\infty$.  This is true for
all such balls $B_r$, so if we fix the ball $B(x,r)$ so that
$r=2|z-y|$, we get that
\begin{equation} \label{eq:morrey1}
 |f(z)-f(y)| \leq
  C|x-y|^{1-n\frac{q}{p}}\sigma(B_{2R})^{(q-1)/p}\|\nabla
  f\|_{L^{p}(B_R,w)}. 
\end{equation}
Since $ C^{\infty}(\Omega)\cap W^{1,p}(\Omega,w)$ is dense in
$W^{1,p}(\Omega,w)$ (see ~\cite[Theorem 1]{gol2009weighted}), a
standard argument shows that inequality~\eqref{eq:morrey1} holds for
every $f\in W^{1,p}\left(\Omega,w\right)$ and almost every
$y,\,z\in B(x_0,R)$.  Consequently, arguing as in the proof
of~\cite[Theorem 9.12]{BrezisBook}, we get that $f$ has a continuous
representative (that is to say, by redefining $f$ on a null set, we
have that $f$ is continuous on $B(x_0,R)$).  Since this is true for all
$x_0\in \Omega$, we have that $f$ has a continous representative on $\Omega$.

To estimate the $RBV^p(\Omega,w)$ norm of $f$, fix any collection
$\D=\{B_k\}_{k=1}^\infty$ of disjoint balls contained in $\Omega$.
Since $w\in A_q$,
			\[ 
			w(B_k)\sigma(B_k)^{q-1}\lesssim r_k^{nq}.  
			\]     
                        Therefore, by \eqref{eq:morrey1} we have that
			\begin{align*}
				\sum_{B_k
                          }{\left(\frac{\osc_{B_k}(f)}{r_k}\right)^{p}w(B_k)}
                          &=\sum_{B_k }{\left(\frac{\osc_{B_k}(f)}{r_k}\right)^{p}w(B_k)\frac{\sigma(B_k)^{q-1}}{\sigma(B_k)^{q-1}}}\\
				&\lesssim \sum_{B_k}\frac{\osc_{B_k}(f)^{p}}{r_k^{p-qn}}\frac{1}{\sigma(B_k)^{q-1}} \\
				& \lesssim \sum_{B_k}\|\nabla f\|_{L^{p}(B_k,w)}^{p} \\
				&\lesssim \|\nabla
                           f\|_{L^{p}(\Omega,w)}^{p}. 
			\end{align*}   
			If we take the supremum over all such
                        collections $\D$, we get that
                        $V_{p}(f;\Omega,w) \lesssim\|\nabla
                        f\|_{L^{p}(\Omega,w)}$. 
                        
			\medskip

                        To prove the converse, suppose $f\in
                        RBV^p(\Omega,w)$.  
                       We will adapt the construction in
                        \cite[Theorem~1.1]{Barza-Lind2015}.  Fix an open set
                        $\Omega_{0}\Subset\Omega$ and fix
                        \[
			0< R<\frac{d\left(\Omega_{0},\Omega^c\right)}{6\sqrt{n}}.
			\]
                        Let $\{\alpha_{k}\}_{k=1}^\infty$ be an
                        enumeration of the lattice
                        $2R\mathbb{Z}^{n}$. Form an index set
                        $K\in \N$ so that the collection of cubes
                        $\D=\{Q_k\}_{k\in K}$ consists of all cubes
                        such that each cube $Q_k$ has center
                        $\alpha_{k}$, side length $2R$, and satisfies
                        $Q_{k}\cap \Omega_{0}\neq \emptyset$.  We
                        claim that
                        \[ \Omega_ 0 \subset \bigcup_{k\in K}{Q_k}\subset
                          \Omega. \]
                        The first inclusion is immediate.  If the
                        second did not hold, then for some $Q\in \D$
                        we would have that
                        $Q\cap \Omega^{c}\neq \emptyset$. Hence, there
                        exists $x\in Q\cap \Omega_{0}$ and
                        $y\in Q\cap \Omega^{c}$ such that
			\[
			d\left(\Omega_{0},\Omega^{c}\right)\leqslant |x-y|\leqslant 2\sqrt{n}R,
			\]       
			which is a contradiction. Denote by
                        $\B=\{B(\alpha_{k},3\sqrt{n}R)\}_{k\in K}$ the
                        collection of balls with center $\alpha_{k}$
                        and radius $3\sqrt{n}R$, and observe that,
                        arguing as before, we have that 
                        $Q_{k}\subset B(\alpha_{k}, 3\sqrt{n}R) \subset
                        \Omega$.   Since the centers of the balls in
                        $\B$ lie on a lattice and they have uniform
                        radii, the collection $\B$ can be
                        partitioned into  $N$ collections
                        $\{\B_{i}\}_{i=1}^{N}$ of disjoint balls,
                        $\B_i\subseteq \B$, where $N$ depends only on
                        the dimension.  For brevity, let
                        $B_k=B(\alpha_{k}, 3\sqrt{n}R)$.

                        Let $\varphi \in C_c^\infty(B(0,1))$ be
                        non-negative, radially descreasing function such that $\|\varphi\|_1=1$.
                        For $R>0$ define the sequence of mollifiers
                        $\varphi_{R}(x)=\frac{1}{R^{n}}\varphi(x/R)$.
                       For all 
                        $x\in Q_k$, $y\in B(0,R)$ implies $x-y\in
                        B_k$.    Moreover, we have that for $1\leq j
                        \leq n$, $D_j \varphi$ is uniformly bounded
                        and has $\int_{B(0,1)} D_j\varphi(x)\dif x =
                        0$  (the latter is well-known and follows,
                        for instance,
                        from the divergence theorem).    Therefore,
                        for each $j$ we can estimate as follows:
			\begin{align*}
				\norm{D_{j}
                          (\varphi_{R}*f)}_{L^{p}(\Omega_{0},w)}^{p} &\leqslant
                                                                      \sum_{k\in
                                                                      K}\int_{Q_{k}}|D_{j}\varphi_{R}*f(x)|^{p}w(x)\dif{x}\\
				& = \sum_{k\in K}\int_{Q_{k}}\abs{\int_{B(0,R)}D_{j}\varphi_{R}(y)(f(x-y)-f(x))\dif{y}}^{p}w(x)\dif{x}\\
				&\lesssim \sum_{k\in K}\int_{Q_{k}}R^{-(n+1)p}(\osc_{B_{k}}(f))^{p}|B(0,R)|^{p}w(x)\dif{x}\\
				&\lesssim \sum_{k\in K}\left(\frac{\osc_{B_k}(f)}{3\sqrt{n}R}\right)^{p}w(Q_{k}) \\
				& \leqslant \sum_{k\in K}\left(\frac{\osc_{B_k}(f)}{3\sqrt{n}R}\right)^{p}w(B_{k})\\
				& = \sum_{i=1}^{N}\sum_{B_{k}\in
                           \B_i}\left(\frac{\osc_{B_k}(f)}{3\sqrt{n}R}\right)^{p}w(B_{k}) \\
				& \lesssim V_{p}(f;\Omega,w)^{p}, 
			\end{align*} 
                        where the last inequality follows since the
                        balls in $\B_i$ are disjoint.

Because $w\in A_p$, the Hardy-Littlewood maximal operator is bounded on
$L^p(\Omega,w)$,  which yields $\|\varphi_R*f\|_{L^p(\Omega_0,w)} \lesssim \|f\|_{L^p(\Omega,w)}$,  since $|\varphi_R*f(x)| \lesssim Mf(x)$.  
                        Thus,
                        $\{\varphi_{R}*f\}$ is a bounded sequence in
                        $W^{1,p}\left(\Omega_{0},w\right)$ 
			which converges pointwise almost everywhere to
                        $f$.  Thus, by Fatou's lemma, we have
                        that
                        \[ \|\nabla f\|_{L^p(\Omega_0,w)} \lesssim
                          V_{p}(f;\Omega,w). \]
                        Since the implicit constants do not depend on
                        $\Omega_0$, we can again apply Fatou's lemma
                        to get that the left-hand side 
                        inequality in ~\eqref{eq:Main_Inequality} holds.
                        Since we also have that $f\in
                        L^p(\Omega,w)$, it follows that  $f\in W^{1,p}\left(\Omega,w\right)$.  
 \end{proof}
                      
		\begin{proof}[Proof of Corollary~\ref{cor:RBV-embed}]
                  For this result it suffices to note that in the proof of the left-hand
                  side inequality in \eqref{eq:Main_Inequality} we did
                  not use the $A_{p}$ condition and this estimate
                  holds for any weight $w$ and for all $p$.  
                \end{proof}

		\begin{proof}[Proof of Theorem \ref{prop:Main-Proposition}]
                  Fix $t>0$.  By the definition of $L_{f}$, for any $x$ for which 
                  $L_{f}(x)>t$ there exists $\Delta_{x}\in \subRn$, $|\Delta_x|\leqslant 1$, such that
			\begin{equation}\label{eq:weak-Inequality}
				|f(x+\Delta_x)-f(x)| \geqslant t|\Delta_x|.
			\end{equation}    
                        Since the balls $\{ B(x,|\Delta_x|)\}$ cover the
                        set $\{ x : L_f(x)>t\}$ and have uniformly
                        bounded radii, by the Vitali
                        covering lemma there exists a subcollection
                        of disjoint balls $\{B_{k}\}_{k=1}^\infty$,
                        $B_{k}:=B(x_{k},|\Delta_{x_{k}}|)$, such that
			\[
			\{x\in \subRn : L_{f}(x)>t\}\subset \bigcup_{k} 5B_{k}.
			\]  
			By Lemma \ref{lemma: GeneralDoublingCon}
                        (applied to balls instead of cubes) we have
			\[
			\left(\frac{|B_{k}|}{|5B_{k}|}\right)^{p}\lesssim \frac{w(B_{k})}{w(5B_{k})}.
			\]       
			Hence,
			\begin{equation}\label{eq:weak-Inequality2}
w(\{x\in \subRn :  L_{f}(x)>t\}) \leq
w\bigg(\bigcup_{k}5B_{k}\bigg) \leq \sum_{k}w(5B_{k}) \lesssim \sum_{k}w(B_{k}).
\end{equation}
			Set $r_{k}=|\Delta_{x_k}|$. If we combine
                        \eqref{eq:weak-Inequality} and
                        \eqref{eq:weak-Inequality2} we get that 
			\begin{multline*}
                          t^{p}w(\{x\in\subRn : L_{f}(x)>t\})
                          \lesssim t^p\sum_{k}w(B_{k}) \lesssim 
                          \\ \lesssim
                          \sum_{k}\left(\frac{\osc_{B_k}(f)}{r_{k}}\right)^{p}w(B_{k})
                          \leqslant V_{p}(f;\Omega,w)^p,
			\end{multline*}
		which is the desired inequality.	
              \end{proof}

 \begin{proof}[Proof of Corollary~\ref{cor:doubling}]
 To see that this result is true, it suffices to note that in the proof of
 Theorem~\ref{prop:Main-Proposition}, we only use the $A_p$ condition
 to invoke Lemma~\ref{lemma: GeneralDoublingCon}, which in turn we
 only use to show that for each $k$, $w(5B_k)$ is uniformly bounded by $w(B_k)$. 
 Thus, it is enough to assume doubling.
\end{proof}

\begin{proof}[Proof of  Corollary~\ref{cor:Almost-Differentiability-W1p}]
Fix $f\in RBV^p(\Omega,w)$; then   by
Proposition~\ref{prop:Main-Proposition},  $|\{x\in \subRn : L_{f}(x)=\infty\}|=0$.
Given any bounded set $E$,
\[ |E| = \int_E w(x)^{1/p}w(x)^{-1/p} \dif x
  \leq \bigg(\int_E w(x)\dif x\bigg)^{1/p}
  \bigg(\int_E w(x)^{1-p'}\dif x\bigg)^{1/p'}; \]
since $w\in A_p$ the second integral is finite.  Therefore, given any
set $E$,  if
$w(E)=0$, it follows by a standard approximation argument that
$|E|=0$.   In particular, we have that $|\{x\in \subRn :
L_{f}(x)=\infty\}|=0$, and so by 
Stepanov's theorem \cite[Theorem 3.1.9]{Federer1969}, $f$ is differentiable almost everywhere.

If $p>nr_w$ and $f\in W^{1,p}(\Omega,w)$, then by Theorem
\ref{theo:RBVpq-W1p}, $f\in RBV^p(\Omega,w)$ and so again is
differentiable almost everywhere. 
\end{proof}

\begin{remark}
  By Corollary~\ref{cor:doubling}, if $w$ is doubling and mutually
  absolutely continuous with respect to Lebesgue measure, then $f\in
  RBV^p(\Omega,w)$ is again differentiable almost everywhere. 
\end{remark}

\begin{proof}[Proof of
  Theorem~\ref{thm:embedding-n-variation-sobolev}]
Given a weight $w$ and $f\in RBV^n(\Omega,w)\cap L^n(\Omega,w)$,  by
Corollary~\ref{cor:RBV-embed}, we have that  $f\in W^{1,n}(\Omega,w)$ and
the gradient estimate holds.  If $w\in A_n$, then by
Theorem~\ref{prop:Main-Proposition}, $f$ is differentiable almost
everywhere.
\end{proof}

                \section{Proof of
                  Theorem~\ref{theo:Riesz_Variable-Exponent}
                }\label{sec:Variable-expoent-section}
			
                In this section, we prove
                Theorem~\ref{theo:Riesz_Variable-Exponent}, which
                  characterizes the  variable exponent
                Sobolev space $W^{1,p(\cdot)}(\Omega)$ in terms of
                the  variable exponent Riesz bounded variation spaces
                $RBV^{p(\cdot)}(\Omega)$.   We first give a formal
                definition of these spaces.  Given a  countable
                collection of disjoint balls $\D=\{B_k\}_{k=1}^\infty$ of radius $r_{k}$ contained in
                $\Omega$, define
			\begin{equation*}
                          V_{\D}^{p(\cdot)}(f;\Omega)
                          =\sum_{k}\left(\frac{\osc_{B_{k}}(f)}{r_k}\right)^{p_{B_k}}
                          \|\chi_{B_{k}}\|_{p(\cdot)}^{p_{B_k}}.  
                              \end{equation*}
Use the Luxemburg norm to define the functional 
			\begin{equation} \label{eq:def_RBV_variable}
                          \|f\|_{RBV_{\D}^{p(\cdot)}(\Omega)}
                          = \inf\left \{ \lambda>0 :  V_{\D}^{p(\cdot)}(f/\lambda;\Omega)\leqslant 1 \right \}.
                        \end{equation}
                        A function $f$ belongs to $RBV^{p(\cdot)}(\Omega)$ if 
			\[
			\|f\|_{RBV^{p(\cdot)}(\Omega)}=\sup_\D\|f\|_{RBV_{\D}^{p(\cdot)}(\Omega)}< \infty,
                      \]
                      where the supremum is taken over all the
                      countable collections $\D$.   While not central
                      to our results, we have that
                      $\|\cdot\|_{RBV^{p(\cdot)}(\Omega)}$ is a
                      seminorm.  This can be proved using the standard
                      arguments used with the Luxemburg construction
                      to prove that the norms in Orlicz spaces and
                      variable Lebesgue spaces have the requisite
                      properties, 
                      see~\cite{cruz-fiorenza-book,MR1113700}.  
The fact that 
                      $\|\cdot\|_{RBV^{p(\cdot)}(\Omega)}$ is a
                      only  a seminorm is immediate, since 
                      $\|f\|_{RBV^{p(\cdot)}(\Omega)}=0$ whenever 
                      $f$ is a constant function. Finally, note that when
                      $\pp=p$ is constant, $1\leq p<\infty$, 
                      $ \|f\|_{RBV^{p(\cdot)}(\Omega)}=
                      V_{p}(f;\Omega)$, so the spaces
                      $RBV^\pp(\Omega)$ and $RBV^p(\Omega)$ are
                      equal.

                      To prove our main results in this section, we
                      first prove an alternative way to estimate the
                      seminorm in $RBV^\pp(\Omega)$.  Given an exponent
                      function $\pp$, a disjoint
                      collection $\D=\{B_k\}_{k=1}^\infty$ of balls
                      contained in $\Omega$, and a sequence
                      of real numbers $\{t_{B_{k}}\}_{B_{k}\in \D}$
                      associated to $\D$,  define the variable
                      exponent sequence space
			\begin{equation*}
                          \ell^{\D,\pp}
                          = \bigg \{\hat{t}=\{t_{B_{k}}\}_{B_{k}\in
                              \D} :
                            \sum_{B_k \in \D }|t_{B_k}|^{p_{B_{k}}}<\infty    \bigg \},
			\end{equation*}  
			and endow it with the Luxemburg norm    
			\[
                          \|\hat{t}\|_{\ell^{\D,\pp}}
                          =\inf \bigg \{ \lambda>0 :
                            \sum_{B_k \in \D }(|t_{B_k}|/\lambda)^{p_{B_{k}}}\leqslant 1 \bigg \},
                          \]
                          for more on variable sequence spaces,
                          see~\cite{diening-harjulehto-hasto-ruzicka2010}.  
                          The following result shows the connection
                          between these sequence spaces and the
                          variable Lebesgue space $L^\pp(\Omega)$
                          (for a proof,
                          see~\cite[Lemma~2.1]{kopaliani2008greediness}).
                          
			\begin{lemma}\label{theo:sequence_Lp_norm_ineq}
                          Let $\pp \in LH(\Omega)$ and
                          $\{e_{B_{k}}\}_{B_{k}\in \D}$ be the
                          canonical base for $\ell^{\D,\pp}$ (i.e., $e_{B_k}$
                          has entry 1 at the index $k$, and $0$
                          otherwise). Then for every disjoint
                          countable collection $\D$ of balls and every
                          real sequence $\{t_{B_k}\}$,
				\begin{equation}
				\bigg\|\sum_{B_k \in \D
                                          }t_{B_{k}}\chi_{B_{k}}\bigg\|_{p(\cdot)}
                                        \approx \bigg\|\sum_{B_k \in \D
                                          }t_{B_k}
                                          \|\chi_{B_k}\|_{p(\cdot)}e_{B_k}\bigg\|_{\ell^{\D,p}}.
				\end{equation}
                                The implicit constants depend only on $\pp$.
                              \end{lemma}
                              
                              Given a disjoint countable collection of
                              balls $\D=\{B_{k}\}_{k=1}^\infty$, define
                              the operator $G_{\D}$ by
			\begin{equation*}
				G_{\D} f(x)=\sum_{B_k\in \D}\frac{\osc_{B_k}(f)}{r_k}\chi_{B_{k}}(x). 
                              \end{equation*}
                              A key fact is that the 
                              $RBV_{\D}^{p(\cdot)}$ norm of $f$ and the
                              $L^{p(\cdot)}$ norm of the operator  $G_{\D} f$ are
                              comparable.
                              
			\begin{proposition}\label{cor:AveragingOp_RBV_equivalence}
				Given $\pp \in LH(\Omega)$, for every function
                                $f$ we have 
				\begin{equation}\label{eq:AveragingOp_RBVpnorm_equivalence}
                                  \|f\|_{RBV^{p(\cdot)}(\Omega)}
                                  \approx \sup_{\D}\|G_{\D}f\|_{L^{p(\cdot)}(\Omega)},
				\end{equation}
                                where the supremum is taken over every countable
                                collection $\D$ of disjoint balls in
                                $\Omega$.  
                                The implicit constants depend only on $\pp$.
                              \end{proposition}

                              \begin{proof}
Fix a collection $\D$. Then by Lemma~\ref{theo:sequence_Lp_norm_ineq},
the definition of $\ell^{\D,\pp}$, and the definition of $RBV_\D^\pp$,
\eqref{eq:def_RBV_variable}, we have that
\begin{multline*}
  \|G_\D f\|_\pp
  = \bigg\| \sum_{B_k \in \D} \frac{\osc_{B_k}(f)}{r_k} \chi_{B_k}
  \bigg\|_\pp \approx \\
  \approx
  \bigg\| \sum_{B_k \in \D} \frac{\osc_{B_k}(f)}{r_k}
  \|\chi_{B_k}\|_\pp
  \bigg\|_{\ell^{\D,\pp}}
  = \|f\|_{RBV_\D^\pp(\Omega)}.
  \end{multline*}
  If we take the supremum over every collection $\D$, we get
  \eqref{eq:AveragingOp_RBVpnorm_equivalence}.
 \end{proof}     
		
Finally, for our proof we need a version of Rubio de Francia
extrapolation into the scale of variable Lebesgue spaces.  More
precisely, we need a version of limited range extrapolation that was
proved in~\cite[Theorem 2.14]{Cruz-Wang_2017}.  This result is stated
in terms of an abstract family of extrapolation pairs $\F = \{
(f,g)\}$. 

 \begin{theorem}\label{theo:interpo_result}
  Given a family of extrapolation pairs $\F$,  let $1<q_{-}<q_+<\infty$ and assume that there exists ${p}$,
   $q_{-}<{p}<q_+,$ such that for every
   ${w}\in A_{{p}/q_-}\cap RH_{(q_+/{p})^{\prime}}$,
   \begin{equation}\label{eq:Extrapol_result2}
 \|f\|_{L^{{p}}({w})}\leqslant C\|g\|_{L^{{p}}({w})}
 \end{equation}
 for every pair $(f,g)\in \mathcal{F}$ such
 that $\|f\|_{L^{{p}}({w})}<\infty$. Given
 $p(\cdot)\in LH(\Omega)$, suppose $q_{-}<p_{-}\leqslant p_{+}<q_{+}$. Then for
 $(f,g)\in \mathcal{F}$ such that $\|f\|_{p(\cdot)}<\infty$,
 \begin{equation}\label{eq:Extrapol_result}
 \|f\|_{p(\cdot)}\leqslant C\|g\|_{p(\cdot)}.
 \end{equation}  
\end{theorem}   
			
\begin{remark}
  We want to emphasize that the family of extrapolation pairs $\F$
  must be chosen so that the left-hand side of
  ~\eqref{eq:Extrapol_result2} and~\eqref{eq:Extrapol_result} are
  finite.  Given this assumption, to prove a desired inequality in general often requires an
  approximation argument; this is the case in our proof below.  On the
  other hand, we do not need to assume that the right-hand side of
  either of these inequalities is finite.
 For a complete discussion of this approach to
extrapolation, see~\cite{Cruz-Martell2011,Cruz-Wang_2017}.  
\end{remark}

\begin{proof}[Proof of Theorem \ref{theo:Riesz_Variable-Exponent}]
First assume that $f\in W^{1,\pp}(\Omega)$.  Let $B\subset \Omega$ be
any ball.  Then by the embedding of variable Lebesgue
spaces on compact domains~\cite[Corollary~2.48]{cruz-fiorenza-book},
\[ W^{1,\pp}(\Omega) \subset W^{1,\pp}(B) \subset W^{1,p_-}(B).  \]
Since $p_->n$, by~\cite[Theorem~1.1]{Barza-Lind2015} the function $f$ has a
continuous representative on $B$.  Since this is true for all balls
contained in $\Omega$, $f$ has a continuous representative on
$\Omega$. 

We will now prove that the right-hand side inequality in~\eqref{eq:
  var-exponent_Riesz} holds:  i.e., that $\|f\|_{RBV^\pp(\Omega)}
  \lesssim \|\nabla f\|_{L^\pp(\Omega)}$.  Fix a
collection $\D=\{B_{k}\}_{k=1}^\infty$ of disjoint balls contained in
$\Omega$. Since the balls are disjoint, given any weight $w$, we obtain
\begin{equation} \label{eqn:G-norm}
  \begin{split}
\|G_{\D}f\|_{L^{p}(w,\Omega)}^{p}
   &=\int_{\Omega}(G_{\D}f(x))^{p}w(x)\dif{x} \\
   &= \sum_{B_k\in\D}\left(\frac{\osc_{B_k}(f)}{r_k}\right)^{p}w(B_k)\\
    &= V_p(f;\Omega,w)^p.
    \end{split}
 \end{equation}
 To apply extrapolation we need that the left-hand side
 of~\eqref{eq:Extrapol_result2} and~\eqref{eq:Extrapol_result} are
 finite.  To ensure this, we define a truncation of the operator
 $G_\D$.  For each $N\geq 1$, define
 $\D_{N}=\{B_{k} : B_{k}\in \D, 1\leqslant k \leqslant N\}$ and the
 truncated operator $G_{D_{N}}f$.  Define the family of extrapolation pairs
 \begin{equation*}
   \mathcal{F}=
   \left\{ (G_{\D_N}f ,\nabla f) :   f\in  C^{\infty}(\Omega)\cap
     W^{1,p(\cdot)}(\Omega),\,
     N\geqslant 1 \right\}.
 \end{equation*}
Fix $\pp\in LH(\Omega)$ such that $n<p_- \leq p_+<\infty$.  Let
$q_-=n$, and fix $q_+>p_+$.  Let $w$ be any weight in $A_{p_-/n}\cap
RH_{(q_+/p_-)'}$  (note that this class is not empty: if
we take $w\in A_{p_-/n}$, by Lemma~\ref{lemma:Ainfty_RHS} we have
$w\in RH_{(q_+/p_-)'}$ for all $q_+$ sufficiently large).  Since $w\in
A_{p_-/n}$, we have that $p_->nr_w$ and $w\in A_{p_-}$.  Fix $f\in C^{\infty}(\Omega)\cap
     W^{1,p(\cdot)}(\Omega)$.  For all $N$,
$G_{\D_N}f$ is compactly supported and in $L^\infty$ with $\|G_{\D_N}\|_{L^{p_-}(w)}
<\infty$, since $w$ is locally integrable.  Thus, by
Theorem~\ref{theo:RBVpq-W1p} and~\eqref{eqn:G-norm}
\[ \|G_{\D_N}f\|_{L^{p_{-}}(\Omega,w)}
  \leqslant C \|\nabla f\|_{L^{p_{-}}(\Omega,w)}. \]
Similarly, we have that
\[ \|G_{\D_N}f\|_\pp \leq \|G_{\D_N}f\|_\infty
  \|\supp(G_{\D_N}f)\|_\pp < \infty.  \]
Therefore, by Theorem~\ref{theo:interpo_result} applied to the family
$\F$, for every $N\geq 1$ and $f\in C^{\infty}(\Omega)\cap
     W^{1,p(\cdot)}(\Omega)$,
\begin{equation*}
  \|G_{\D_N}f\|_{L^{p(\cdot)}(\Omega)}
  \leqslant C \|\nabla f\|_{L^{p(\cdot)}(\Omega)}.
\end{equation*}
If we apply  Fatou's lemma in the variable Lebegue spaces (see
\cite[Theorem 2.61]{cruz-fiorenza-book}) we have  that for all $f\in
C^{\infty}(\Omega)\cap W^{1,p(\cdot)}(\Omega)$,
 \begin{equation} \label{eqn:dense-set}
   \| G_{\D}f\|_{L^{p(\cdot)}(\Omega)}
   \leqslant C \|\nabla f\|_{L^{p(\cdot)}(\Omega)}.
 \end{equation}

 To complete the proof, fix $f\in W^{1,p(\cdot)}(\Omega)$.  Recall
 that $C^{\infty}(\Omega)\cap W^{1,p(\cdot)}(\Omega)$ is dense in
 $W^{1,p(\cdot)}(\Omega)$ (see \cite[Theorem
 6.14]{cruz-fiorenza-book}).  Fix  a sequence $\{f_{k}\}_{k=1}^\infty$
 of functions in $C^{\infty}(\Omega)\cap W^{1,p(\cdot)}(\Omega)$ that converges
 to $f$ in the $W^{1,p(\cdot)}$ norm. Then we can find a subsequence (still denoted by $f_k$)  
 which converges pointwise to $f$
 (see \cite[Theorem 2.67]{cruz-fiorenza-book}).  Therefore, by Fatou's
 lemma in the variable Lebesgue spaces and by
 inequality~\eqref{eqn:dense-set},
 \begin{equation*}
   \|G_{\D}f\|_{L^{p(\cdot)}(\Omega)}
   \leqslant  \liminf_{k\rightarrow \infty}\|G_{\D}f_k\|_{L^{p(\cdot)}(\Omega)}
   \lesssim \lim_{k\rightarrow\infty} \|\nabla f_k\|_{L^{p(\cdot)}(\Omega)}
   = \|\nabla f\|_{L^{p(\cdot)}(\Omega)}.
 \end{equation*}
 The desired inequality now follows from Proposition~\ref{cor:AveragingOp_RBV_equivalence}.             

\medskip			

We will now prove the converse.  Fix $f\in RBV^\pp(\Omega) \cap
L^\pp(\Omega)$.  We will prove that  $f\in W^{1,\pp}(\Omega)$ and the left-hand side inequality in ~\eqref{eq:
  var-exponent_Riesz} holds:  i.e., that $\|\nabla f\|_{L^\pp(\Omega)}
  \lesssim \|f\|_{RBV^\pp(\Omega)}$. 

  Let $\Omega_0\Subset  \Omega$; then we have the embedding
 \begin{equation*}
   RBV^{p(\cdot)}(\Omega)\subset  RBV^{p(\cdot)}\left(\Omega_0\right)
   \subset RBV^{p_{-}}\left(\Omega_0\right).
 \end{equation*}
 The first inclusion is immediate; the second one does not follow
 directly from the definition of $RBV^\pp$, but is a consequence of
 Corollary~\ref{cor:AveragingOp_RBV_equivalence} and the embedding of
 variable Lebesgue spaces on compact
 domains~\cite[Corollary~2.48]{cruz-fiorenza-book}.  Since $p_->n$,
 again by  \cite[Theorem
 1.1]{Barza-Lind2015} we have that
 $RBV^{p(\cdot)}(\Omega)\subset
 W^{1,p_{-}}\left(\Omega_0\right)$. Thus, weak derivatives are well defined in
 $RBV^{p(\cdot)}(\Omega)$.

 We now define  the family of extrapolation pairs
 \[
   \mathcal{F}=\big \{ (\langle\nabla f\rangle_{N}, G_{\D}f ) : f\in
     RBV^{p(\cdot)}(\Omega),\ N\geqslant 1, \D \big\},
 \] 
 where $\langle\nabla f\rangle_{N}=\min\{N,|\nabla f|\}$ and
 $\mathcal{D}$ is a countable family of disjoint
 balls contained in $\Omega$ that depends on $f$ and $N$; the exact
 choice of $\D$ will be made below.  Because $w$ is locally integrable, 
 $\| \langle\nabla f\rangle_{N} \|_{L^{p_-}(\Omega,w)}<\infty$;
 similarly, since $p_+<\infty$, $\| \langle\nabla f\rangle_{N}
 \|_{\Lp(\Omega)}<\infty$. 
We now argue as we did in the proof above, fixing $q_-,\,q_+$, and 
$w\in A_{\frac{p_{-}}{n}}\cap RH_{(\frac{q_{+}}{p_{-}})^{\prime}}$.
If $f\in RBV^{p_-}(\Omega,w)$, then by Theorem~\ref{theo:RBVpq-W1p}
and inequality~\eqref{eqn:G-norm}
\begin{equation} \label{eqn:saturate}
\| \langle\nabla f\rangle_{N} \|_{L^{p_-}(\Omega,w)}
  \leq \| \nabla f \|_{L^{p_-}(\Omega,w)}
  \lesssim V^{p_-}(f,\Omega,w)
  \leq 2\|G_\D f\|_{L^{p_-}(\Omega,w)}, 
\end{equation}
where we choose $\D$ to saturate the supremum used to define
$V^{p_-}$.  On the other hand, if $f\not\in RBV^{p_-}(\Omega,w)$, then
$V^{p_-}(f,\Omega,w)=\infty$, so we can choose $\D$ so that
$\|G_\D f\|_{L^{p_-}(\Omega,w)}$ is either infinite or arbitrarily large; in particular, we can fix $\D$ so
that
$\| \langle\nabla f\rangle_{N} \|_{L^{p_-}(\Omega,w)} \lesssim \|G_\D
f\|_{L^{p_-}(\Omega,w)}$ with the same constant as
in~\eqref{eqn:saturate}.  Since the hypotheses of
Theorem~\ref{theo:interpo_result} are
satisfied for any pair  $(\langle\nabla f\rangle_{N}, G_{\D}f )$ in
$\mathcal{F}$, we conclude that
 \begin{equation*}
   \|\langle\nabla f\rangle_{N} \|_{L^{p(\cdot)}(\Omega)}
   \lesssim \| G_{\D} f\|_{L^{p(\cdot)}(\Omega)}
   \lesssim \|f\|_{RBV^\pp(\Omega)};
 \end{equation*}
 the last inequality holds by
 Corollary~\ref{cor:AveragingOp_RBV_equivalence}. 
 If we again apply  Fatou's lemma in the variable Lebesgue spaces, we get
 \[
   \|\nabla f\|_{L^{p(\cdot)}(\Omega)}
   \leqslant \liminf_{N\rightarrow\infty} \|\langle \nabla
   f\rangle_{N}\|_{L^{p(\cdot)}(\Omega)}
   \lesssim  \|f\|_{RBV^{p(\cdot)}(\Omega)},
 \]      
which completes the proof.
\end{proof}


\begin{remark}
  Though our goal was to prove
  Theorem~\ref{theo:Riesz_Variable-Exponent} via extrapolation, we
  want to note that we can also prove the second half of this result
  by adapting the proof of Theorem \refeq{theo:RBVpq-W1p} to the
  variable exponent setting. Here we sketch the proof.

  We will follow the same construction and notation from the proof of
  Theorem \refeq{theo:RBVpq-W1p}. First, by homogeneity of the norm
  and by Corollary \refeq{cor:AveragingOp_RBV_equivalence}, without
  loss of generality, we may assume that
  $\|f\|_{RBV^{p(\cdot)}(\Omega)}=c_0$, for some constant that is
  sufficiently small that
  $\|G_{\D} f\|_{L^{p(\cdot)}(\Omega)}\leqslant 1$ for any collection
  of balls $\D$.  Therefore, by~\cite[Corollary
  2.22]{cruz-fiorenza-book},
  $\rho_{p(\cdot),\Omega}\left(G_{ \D}f\right)\leqslant 1$.  Since
  $p_{+}<\infty$ and $\mathcal{B}_{i}$ is a disjoint collection of
  balls, we get the estimate
  \begin{align*}
    \rho_{p(\cdot),\Omega_0} \left( D_{j} (\varphi_{R}*f)\right)
    &\leqslant \sum_{k\in K}\int_{Q_{k}}|D_{j}\varphi_{R}*f(x)|^{p(x)}\dif{x}\\
    & = \sum_{k\in K}\int_{Q_{k}}
      \abs{\int_{B(0,R)}D_{j}\varphi_{R}(y)(f(x-y)-f(x))\dif{y}}^{p(x)}\dif{x}\\
    &\lesssim \sum_{k\in K}\int_{Q_{k}}
      R^{-(n+1)p(x)}(\osc_{B_{k}}(f))^{p(x)}|B(0,R)|^{p(x)}\dif{x}\\
    &\lesssim \sum_{k\in K}\int_{B_k}
      \left(\frac{\osc_{B_k}(f)}{3\sqrt{n}R}\right)^{p(x)}\dif{x} \\
 & = \sum_{i=1}^{N}\sum_{B_{k}\in \B_i}\int_{B_k}\left(\frac{\osc_{B_k}(f)}{3\sqrt{n}R}\right)^{p(x)}\dif{x}\\
 &=
   \sum_{i=1}^{N}\int_{\Omega}\left(G_{\mathcal{B}_{i}}f\right)^{p(x)}\dif{x}
    \\
    & \leqslant
      \sum_{i=1}^{N}\rho_{p(\cdot),\Omega}\left(G_{\mathcal{B}_{i}}f\right)
    \\
    & \leqslant N.
 \end{align*}
 So,
 $\|D_{j} (\varphi_{R}*f)\|_{L^{p(\cdot)}(\Omega_{0})}$ is uniformly bounded.
 Since $\pp \in LH(\Omega)$, the Hardy-Littlewood maximal operator is
 bounded on $L^\pp(\Omega)$, so
 $\|\varphi_{R}*f\|_{L^{p(\cdot)}(\Omega_{0})}$ is bounded.  
 Hence, the sequence $\{\varphi_{R}*f\}_{R>0}$ is uniformly bounded in
 $W^{1,p(\cdot)}(\Omega_{0})$, and $\{\varphi_{R}*f\}_{R>0}$ converges
 pointwise to $f$.  We can adapt the proof
 of~\cite[Theorem~1.32]{Heinonen1993} to the variable exponent setting
 (again using the fact that $\pp \in LH(\Omega)$ and $p_+<\infty$) to
 conclude that $f\in W^{1,p(\cdot)}(\Omega_{0})$ for every
 $\Omega_{0}\Subset  \Omega$.  Then by Fatou's lemma, we get
 that $f\in W^{1,p(\cdot)}(\Omega)$.
 \end{remark}

 \begin{remark}
   Given the previous remark, it would be interesting to have a direct
   proof of the first half of
   Theorem~\ref{theo:Riesz_Variable-Exponent} that did not depend on
   extrapolation.  We conjecture that such a proof is possible, though
   we have not been able to find it.  A starting point would be the
   variants of the Morrey inquality that hold for variable Sobolev
   spaces, see, for example,~\cite[Theorem 6.36]{cruz-fiorenza-book}.
 \end{remark}

 \section*{Acknowledgments}
 The first author is partially supported by a Simons Foundation Travel
 Support for Mathematicians Grant. The second and third authors were
 supported by a Research Start-up Grant of United Arab Emirates
 University, UAE, via Grant G00002994.

			\bibliographystyle{plain}

			\bibliography{Weighted-RBVp_R_n}

                        \end{document}